\documentclass[11pt]{article}
\usepackage{amssymb}
\usepackage{epsfig}
\usepackage{graphicx}
\usepackage{latexsym}
\usepackage{amsmath}
\usepackage{amsthm}
\usepackage{cases}
\usepackage{bm}
\usepackage{color}
\newtheorem{lem}{Lemma}

\newtheorem{problem}{Problem}
\newtheorem{theorem}{Theorem}%[section]
\newtheorem{assumption}{Assumption}%[section]
\newtheorem{remark}{Remark}%[section]
\newtheorem{definition}{Definition}%[section]
\title{ \vspace{1cm}{\bf An Approach to Mismatched Disturbance Rejection Control for Continuous-Time Uncontrollable  Systems
}}

\begin{document}
\baselineskip 16pt
\date{}
\begin{titlepage}
  \maketitle \vspace{0.5cm}

\begin{center}
{Shichao Lv$^1$,  Hongdan Li$^1$, Kai Peng$^2$ ,  Shihua Li$^3$, Huanshui Zhang$^{1*}$  }

$^1$ {College of Electrical Engineering and Automation, Shandong University of Science and Technology, Qingdao, Shandong, P.R.China 266590.} 

$^2$ {School of Power and Energy, Northwestern Polytechnical University, Xi'an, Shaanxi, P.R.China 710072.}

$^3$ {School of Automation, Southeast University, Nanjing, Jiangsu, P.R.China 210096.}
\end{center}

\begin{center}
{\bf Abstract}
\end{center}
This paper focuses on optimal mismatched disturbance rejection control for linear continuous-time uncontrollable systems. Different from previous studies, by introducing a new quadratic performance  index to transform the mismatched disturbance rejection control into a linear quadratic tracking problem, the regulated state can track a reference trajectory and minimize the influence of disturbance. The necessary and sufficient conditions for the solvability and the disturbance rejection controller are obtained by solving a forward-backward differential equation over a finite horizon. A sufficient condition for system stability is obtained over an infinite horizon under detectable condition. This paper details our novel approach for transforming disturbance rejection into a linear quadratic tracking problem. 
The effectiveness of the proposed method is provided with two examples to demonstrate.


%{Corresponding author: Huanshui Zhang (Email addresses: hszhang@sdu.edu.cn). 
%}
\end{titlepage}

\pagestyle{plain} \setcounter{page}{2}
\section{Introduction}
\label{sec:introduction}
%With the growing interest in high-precision control, the implementation of disturbance rejection techniques is generally required in controller design. Therefore, disturbance rejection is a fundamental issue in automatic control. The disturbances in the system are classified into matched and mismatched disturbances according to their relationship with the control input. Different disturbance rejection controllers handle disturbances using different schemes. Most studies have focused on matched disturbance rejection through active disturbance rejection control (ADRC) \cite{GUO20132911,4796887,5572931,ZHAO2014882}, disturbance observer-based control \cite{chen2015disturbance,li2014disturbance}, and sliding mode control (SMC) \cite{shtessel2014sliding,young1999control}.

The need to implement disturbance rejection techniques in controller design has become a fundamental issue in automatic control with the growing interest for high-precision control. According to the relationship with the control input, the disturbance in the system can be divided into matched and mismatched disturbances. Different disturbance rejection controllers handle disturbances using different schemes. Most of the previous studies focused on matched disturbance rejection control such as disturbance observer-based control \cite{chen2015disturbance,li2014disturbance}, active disturbance rejection control (ADRC) \cite{GUO20132911,4796887,5572931,ZHAO2014882} and sliding mode control (SMC) \cite{shtessel2014sliding,young1999control}.

%Comparatively, the rejection of mismatched disturbances is more challenging. Mismatched disturbances extensively exist in the real world. Many practical systems such as permanent magnet synchronous motors, roll autopilots for missiles, and flight control systems are affected by mismatched disturbances \cite{chen2003nonlinear,chwa2004compensation,mohamed2007design}. In contrast to matched disturbances, these disturbances act on a system through a different channel than the control input, or the effects of these disturbances cannot be equivalently transformed into input channels. As a result, regardless of the type of control scheme employed, eliminating the influence of mismatched disturbances on the system state may be impossible \cite{isidori1985nonlinear}. Therefore, a practical approach is to eliminate the effects of mismatched disturbances from certain variables of interest representing the regulated state.

By contrast, rejecting mismatched disturbances is more challenging. Mismatched disturbances widely exist in many practical systems, such as roll autopilots for missiles, permanent magnet synchronous motors, and flight control systems are affected by mismatched disturbances \cite{chen2003nonlinear,chwa2004compensation,mohamed2007design}. Compared with matched disturbances, the effects of these mismatched disturbances cannot be equivalently converted into input channels acting on the system so that they cannot be directly eliminated by the input. Therefore, no matter what control scheme employed, it is impossible to eliminate the influence of mismatched disturbance on the system state \cite{isidori1985nonlinear}. Therefore, a practical approach is to remove the effects of mismatched disturbances in some variables of interest describing the regulated state.

There are several methods \cite{castillo,chen2016adrc,ginoya2013sliding,2012Generalized,9339876,yang2012nonlinear,6129407} for dealing with mismatched disturbances. For a nonlinear system with mismatched disturbances, a novel SMC scheme based on a generalized disturbance observer was presented in \cite{ginoya2013sliding,6129407}. This scheme can reject mismatched disturbances in steady-state controlled outputs. In contrast to \cite{ginoya2013sliding,6129407}, \cite{yang2012nonlinear} treated mismatched disturbances in a multi-input multi-output system with arbitrary disturbance relative degrees. In a linear system with mismatched disturbances, a generalized extended state observer-based control (GESOBC) method was proposed for linear controllable systems to eliminate mismatched disturbances in the steady-state controlled output \cite{2012Generalized}. Similar to \cite{2012Generalized}, \cite{castillo} weakened the restriction of disturbances and improved the disturbance rejection effect by introducing high-order derivatives of disturbances.
However, previous works have mainly focused on controllable systems and these methods are not applicable to uncontrollable linear systems. Furthermore, the above studies do not consider the balance between control input energy costs and disturbance rejection. In other words, their methods of disturbance rejection control are ineffective in minimizing the cost functional.

This paper focuses on the mismatched disturbance rejection control of linear continuous-time systems. The core of this problem is the design of the controller that attenuates or eliminates the effects of mismatched disturbances on the regulated state. To this end, we introduce a novel quadratic performance index so that the regulated state can track the reference trajectory and minimize the effect of disturbances. In this way, we transform the mismatched disturbance rejection control into a linear quadratic tracking (LQT) problem. It should be noted that the performance index of this mismatched disturbance rejection control problem is different from the standard LQT problem. The new performance index enables the regulated state can track a reference trajectory and minimize the influence of disturbance. To the best of our knowledge, this is a novel approach to mismatched disturbance rejection control using LQT control.

%The main contributions of this study can be summarized as follows. For a finite horizon, the necessary and sufficient condition for the existence of a disturbance rejection controller is derived in terms of the Riccati difference equation. Additionally, an analytical expression of the disturbance rejection controller is obtained. By using the decoupling technique, we provide a solution for the forward-backward differential equations (FBDEs) obtained by applying the maximum principle. For an infinite horizon, we provide the sufficient condition of stabilization based on the general algebraic Riccati equation (GARE) with a pseudo-inverse matrix.

The main contributions of this paper are summarized as follows. Firstly, the necessary and sufficient conditions for the existence of the disturbance rejection controller are derived from the Riccati differential equation over a finite horizon, and the analytical expression of the controller is obtained.
%\textcolor{red}{By using the decoupling technique, we provide a solution for the forward-backward difference equations (FBDEs) obtained by applying the maximum principle.}
%By using decoupling techniques, we obtain a forward-backward differential equations (FBDEs) that provides a solution by applying the maximum principle. 
Secondly, in contrast to \cite{castillo,gandhi2020hybrid,2012Generalized} which requires the system to be controllable, we derived stabilization results under the detectability condition alone. We provide the sufficient condition of stabilization over an infinite horizon based on the generalized algebraic Riccati equation (GARE) with a pseudo-inverse matrix. And we further extend the proposed method to receding-horizon control so that it can handle measurable disturbance in real time. Finally, we demonstrate the effectiveness and feasibility of the method by comparing the simulations of proportional-integral-derivative (PID) control in an uncontrollable numerical example, PID and GESOBC methods in a permanent magnet direct current (PMDC) system.

%In addition, we demonstrate the effectiveness and feasibility of the method by comparing the simulations of the method with proportional-integral-derivative (PID) control in a uncontrollable numerical example and an permanent-magnet direct current (PMDC) system.

%The remainder of this paper is organized as follows. In Section \ref{sec2}, the problem of linear quadratic mismatched disturbance rejection control for discrete-time systems is introduced. Controller design and analysis are presented in Section \ref{sec3}. In Section \ref{sec4}, examples are provided to demonstrate the effectiveness of the proposed method. Our conclusions are summarized in the final section.

The remainder of this paper is organized as follows. The mismatched disturbance rejection control problem for linear continuous-time system and some necessary lemmas and definitions are introduced in  Section \ref{sec2}. In Section \ref{sec3}, the design and analysis of the controller over a finite horizon are presented. The stability of the system under the proposed controller over an infinite horizon is analyzed in Section \ref{sec4}. Two examples are provided to demonstrate the effectiveness of the proposed method in Section \ref{sec5}. The final section is our conclusion. And the proofs are in the Appendix.

Notation: $\mathbb{R}^{n}$ represents the $n$-dimensional Euclidean space; the superscripts $'$, $^{-1}$, $^{\dagger}$, and $\Vert \cdot \Vert$ represent the transpose, inverse, pseudo-inverse, and  2-norm of a matrix, respectively; a symmetric matrix $M > 0$ ($\ge$ 0) is positive definite (positive semi-definite); $I$ denotes the unit matrix; $O$ denotes the zero matrix; and $\rho(\cdot)$ denotes a matrix eigenvalue.

\section{Problem Formulation}\label{sec2}
%\subsection{Problem Formulation}
We consider the linear continuous-time system with mismatched disturbances which are described as follows. 
\begin{equation}\label{eq.1}
\begin{aligned}
&\dot{x}_t={Ax}_t+{B} {u}_t +{E}d_t, x_0=x,
%&{y}_m={C}_m {x},\\
%&{y}_{o}={C}_o {x},
\end{aligned}
\end{equation}
where ${x}_t\in{\mathbb{R}}^n$, ${u}_t\in \mathbb{R}^m$, and ${d}_t\in \mathbb{R}^q$ are the state, control input, and disturbance, respectively. $x\in{\mathbb{R}}^n$ is initial value. $ {A}\in\mathbb{R}^{n\times n}$, ${B}\in\mathbb{R}^{n\times m}$, and $ {E}\in\mathbb{R}^{n\times q}$ are known deterministic coefficient matrices. 

\begin{remark}\label{rem.1}
In \eqref{eq.1}, $d_t$ expresses a known disturbance. The $d_t$  includes disturbances for which disturbance models are available \cite{yang1994disturbance,10.1115/1.2896180} and measurable disturbances \cite{9050185,GUO201941}. The problem that a disturbance is only measurable at the current time can be solved by extending receding-horizon control, see the end of Section \ref{sec4} for details.
\end{remark}
\begin{remark}\label{rem.2}
A mismatched disturbance means that the matching condition ($B\Gamma = E$ for some $\Gamma$) does not hold, which implies that disturbance affect the system through the different channel as the control input or the effects of disturbance cannot be equivalently transformed into input channels.
\end{remark}

In order to make the regulated state of the system optimally track the reference based on the compensation for the disturbance, we transform the mismatched disturbance rejection problem of the linear continuous-time system \eqref{eq.1} into an LQT problem.

%\textcolor{blue}{The mismatched disturbance rejection problem for the linear discrete-time system in \eqref{eq.1} is transformed into an LQT problem such that the regulated state of the system optimally tracks the reference based on compensation for the disturbance.}

Therefore, we define the following cost functional:
\begin{align}\label{eq.2}
J_T=&\frac{1}{2}\int_{0}^{T}[(x_t-r)'Q(x_t-r)+({B} {u}_t +{E}d_t)'R({B} {u}_t +{E}d_t)]\mathrm{d}t+\frac{1}{2}(x_T-r)'P_T(x_T-r),
\end{align}
%where $Q$, $R \in\mathbb{R}^{n\times n}$ are semi-positive definite.
where $Q={{c}_o}'\bar Q{c}_o$ ($c_o\in\mathbb{R}^{l\times n}$, $\bar Q\in\mathbb{R}^{l\times l}$ is a semi-positive definite weight matrix), $x_T$ is the terminal state, and $Q, R, P_{T}\in\mathbb{R}^{n\times n}$ are semi-positive definite.

%, the following function representing the difference between the regulated and reference states can be obtained.
%
%where $c_ox_{t}\in\mathbb{R}^{l}$ denotes the regulated state, ${c}_or$ ($r\in\mathbb{R}^{n}$) represents the reference for the desired regulated state, and $\bar Q\in\mathbb{R}^{l\times l}$ is a semi-positive definite weight matrix.
%On the basis of \eqref{f6.1},

\begin{remark}
The cost functional consists of two parts. The first part
\begin{align}\label{f6.1}
(x_{t}-r)'Q(x_{t}-r)=({c}_ox_{t}-{c}_or)'\bar Q({c}_ox_{t}-{c}_or)
\end{align}
represents the error between the regulated state $c_ox_{t}$ and the reference state ${c}_or$ such that the regulated state tracks the reference. The second part $(Bu_{t}+Ed_{t})'R(Bu_{t}+Ed_{t})$ is the sum of the control input and disturbance for the control input to compensate the disturbance.
\end{remark}
%The proposed cost function can cover both matching and mismatching cases. The invertible matrix $B$ is the simplest case. Clearly, $Bu_{t}+Ed_{t}=0$ has a solution. When $B$ is irreversible, it is necessary to discuss whether a solution to $Bu_{t}+Ed_{t}=0$ exists. When $B\Gamma=E$ for some $\Gamma$ is satisfied, then a solution for $Bu_{t}+Ed_{t}=0$ exists (ADRC is suitable for this type of scenario). Additionally, it is more difficult to handle the case where $\Gamma$ does not satisfy $B\Gamma=E$, which is the main focus of this study.
\begin{remark}
The proposed cost functional \eqref{eq.2} can be applied to both matching and mismatching cases. Obviously, the invertible matrix $B$ is the simplest case, where $Bu_{t}+Ed_{t}=0$ has a solution. In case of $B$ is irreversible, it needs to discuss whether a solution of $Bu_{t}+Ed_{t}=0$ exists. A solution with $Bu_{t}+Ed_{t}=0$ exists when $B\Gamma=E$ for some $\Gamma$ (ADRC is suitable for this type of scenario). The most complicated case is dealing with $\Gamma$ not satisfying $B\Gamma=E$, i.e. mismatching case, which is the main focus of this study.
\end{remark}

\begin{problem}\label{problem}
Find the control $u_t$ such that the effects of the disturbance $d_t$ are minimized, that is, $Bu_t+Ed_t$ is minimized, and the state $x_t$ tracks the reference $r$.
\end{problem}

\section{Optimization over a Finite Horizon}\label{sec3}
%To facilitate further analysis, we first review some preliminaries for subsequent partial results.
First, the following lemma is the key to the solvability of Problem \ref{problem}.
\begin{lem}\cite{qi2018stabilization}\label{l1}
Problem \ref{problem} is uniquely solvable if and only if the following FBDEs have a unique solution,
%\begin{numcases}{}
%B'RBu_t+B'REd_t+B'\lambda_t=0,\label{f2.004}\\
%\dot{\lambda}_t=Q(r-x_t)-A'\lambda_t,\label{f2.005}\\
%\dot{x}_t={Ax}_t+{B} {u}_t +{E}d_t,\label{f2.006}\\
%\lambda_{T}=P_{T}(x_{T}-r),\label{f2.007}\\
%x_0=x.\label{f2.008}
%\end{numcases}
\begin{eqnarray}
\left\{
\begin{array}{lll}
B'RBu_t+B'REd_t+B'\lambda_t=0\\
\dot{\lambda}_t=Q(r-x_t)-A'\lambda_t\\
\dot{x}_t={Ax}_t+{B} {u}_t +{E}d_t\\
\lambda_{T}=P_{T}(x_{T}-r)\\
x_0=x.
\end{array}
\right.\label{f2.10}
\end{eqnarray}
\end{lem}
The goal of the disturbance rejection controller design in this paper is to reduce the effect of mismatched disturbance and make the regulated state of the system follow the reference. To obtain this result, the main technique is to solve the FBDEs \eqref{f2.10}. 

For deriving the solution of Problem 1, we define the following generalized Riccati differential equation (GRDE).
Denote $P_t$ and $M_t$ with the terminal time $T$ as $P_t(T)$ and $M_t(T)$, respectively. Under the regular condition that
\begin{align}\label{2.18}
\Upsilon\Upsilon^{\dag}M_t(T)=M_t(T),
\end{align}

we introduce the GRDE:
\begin{align}\label{2.19}
\dot{P}_t(T)=M'_t(T)\Upsilon^{\dag}M_t(T)-P_t(T)A-A'P_t(T)-Q,
\end{align}
where
\begin{align}
\Upsilon&=B'RB,\label{2.2011}\\
M_t(T)&=B'P_t(T)\label{2.20}
\end{align}
with a terminal value $P_{T}(T)=0$.

\begin{remark}\label{rem0}
For the special case $\Upsilon=B'RB>0$ in \eqref{2.18}, the corresponding Riccati equation is
\begin{align}\label{2.9}
\dot{P}_t=M'_t\Upsilon^{-1}M_t-P_tA-A'P_t-Q,
\end{align}
where
\begin{align}\label{2.2002}
M_t=B'P_t.
\end{align}
%对于公式（2.5）中特殊的情况r>0，则其对应的riccati方程为：
\end{remark}

%\section{Main Results}
%The goal of controller design is to develop a stabilizing disturbance rejection controller to force the regulation state of the system to follow the reference.
%For the linear system \eqref{eq.1} with the cost function in \eqref{eq.2}, the LQT problem can be solved using the AREs derived from the maximum principle to obtain a disturbance rejection controller.
%\section{Optimization over a Finite Horizon}\label{sec3}

Based on the lemma and definition in the preliminaries, the following theorem is given.

\begin{theorem}\label{t1}
Problem \ref{problem} exists a unique solution $u_t$ if and only if $\Upsilon>0$ in \eqref{2.2011}. In that case, the optimal controller is given by
\begin{align}\label{eq.3}
u_t=-\Upsilon^{-1}M_tx_t-\Upsilon^{-1}h_t,
\end{align}
where $h_{t}$ satisfies the following equation:
\begin{equation}\label{eq.4}
\left\{
\begin{aligned}
%&\dot{P}_t=M'_t\Upsilon^{-1}M_t-P_tA-A'P_t-Q,\\
&h_t=B'REd_t+B'f_t,\\
&\dot{f}_t=M'_t\Upsilon^{-1}h_t-A'f_t-P_tEd_t+Qr,\\
%&\Upsilon=B'RB,\\
%&M_t=B'P_t,\\
&f_T=-P_Tr.
\end{aligned}
\right.
\end{equation}
The optimal cost functional is given as
\begin{align}\label{2.41}
J^*=\frac{1}{2}x_0'Px_0+{x}'_0f_0+H_0,
\end{align}
where ${H}_t$ is the solution of
\begin{align}\label{eq.3.11}
\dot{H}_t=\frac{1}{2}h'_t\Upsilon^{-1}h_t-d'_tE'f_t-\frac{1}{2}r'Qr-\frac{1}{2}d'_tE'REd_t
\end{align}

with final condition $H_T=r'P_Tr$.
\end{theorem}
\begin{proof}
The proofs of Theorem \ref{t1} can be found in Appendix A.
\end{proof}
\begin{remark}\label{r11}
The obtained controller \eqref{eq.3} is consistent with that by the existing method. Due to the consideration of the optimality of the disturbance rejection control, the difference from the existing method is that the obtained control parameters of the controller are derived from the Riccati equation instead of directly calculated by the system coefficients.
\end{remark}
\begin{remark}\label{r6}
It is worth noting that disturbances in the regulated state can be optimally eliminated by the controller proposed in Theorem \ref{t1}. This result is novel to our knowledge.
\end{remark}

%\begin{align}\label{2.58}
%&\dot{f}_t=(M'_t\Upsilon^{-1}B'f_t-A')f_t+(M'_t\Upsilon^{-1}B'RE-P_tE)d_t+Qr,\\
%&f_t=e^{(M'_t\Upsilon^{-1}B'f_t-A')}f_0+\int_{0}^{T}e^{(M'_t\Upsilon^{-1}B'f_t-A')(t-\tau)}[(M'_t\Upsilon^{-1}B'R-P_t)Ed_\tau+Qr]\mathrm{d}\tau,\\
%&h_t=B'f_t+B'REd_t.
%\end{align}

\section{Stabilization over an Infinite Horizon}\label{sec4}
%We perform the analysis based on the definitions and lemmas of the preliminaries.
%We analyze on the basis of the preliminaries definition and lemma. 
In order to facilitate the analysis of the stability of the system over an infinite horizon, on the basis of Theorem \ref{t1}, we first give the following lemma.
\begin{lem}\label{l5}
Assume that the GRDE \eqref{2.19}–\eqref{2.20} yields
a solution. Then, Problem \ref{problem} has a solution expressed as
\begin{align}\label{2.201}
u_t=-\Upsilon^{\dag}M_t(T)x_t-\Upsilon^{\dag}h_t(T),
\end{align}
where $h_t(T)$ satisfies the following equation:
\begin{equation}\label{2.21}
\left\{
\begin{aligned}
&h_t(T)=B'f_t(T)+B'REd_t,\\
&\dot{f}_t(T)=M'_t(T)\Upsilon^{\dag}h_t(T)-A'f_t(T)-P_t(T)Ed_t+Qr,\\
&f_T(T)=-P_T(T)r.
\end{aligned}
\right.
\end{equation}
The optimal cost function is given as
\begin{align}\label{2.22}
J_T^*=&\frac{1}{2}x'_0P_0(T)x_0+{x}'_0f_0(T)+H_0(T),
\end{align}
where 
\begin{align}\label{2.23}
\dot{H}_t(T)=\frac{1}{2}h'_t(T)\Upsilon^{\dag}h_t(T)-d'_tE'f_t(T)-\frac{1}{2}r'Qr-\frac{1}{2}d'_tE'REd_t
\end{align}
with terminal condition $H_T(T)=r'P_T(T)r$.
\end{lem}
\begin{proof}
Following the line of Theorem 1, the result can be derived similarly under the regular condition \eqref{2.18}. To avoid redundancy,
we have omitted this proof here.
\end{proof}

Accordingly, we consider the cost functional over an infinite horizon as follows
\begin{align}\label{2.25}
J=&\lim_{T\to \infty}\frac{1}{2T}\int_{0}^{T}[(x_t-r)'Q(x_t-r)+({B} {u}_t +{E}d_t)'R({B} {u}_t +{E}d_t)]\mathrm{d}t.
\end{align}

We now introduce certain definitions and assumptions.
\begin{definition}
The system $(A, Q^{\frac{1}{2}} )$
\begin{equation}\label{2.26}
\left\{
\begin{aligned}
&\dot{x}_t=Ax_t,\\
&y_t=Q^{\frac{1}{2}}x_t,
\end{aligned}
\right.
\end{equation}
is detectable if for any $T\ge0$, the following holds:
\begin{align}
y_t=0,\forall0\leq t \leq T \Rightarrow \lim_{t\to \infty}x_t=0.\notag
\end{align}
\end{definition}
\begin{assumption}\label{a51}
$d_{t}, t\geq 0$ is bounded and $\lim_{t\rightarrow \infty}d_{t}=d.$
\end{assumption}
\begin{assumption}\label{a52}
 $(A,\sqrt{Q})$ is detectable.
\end{assumption}
We define the GARE as follows
\begin{align}
0=&PA+A'P+Q-M'\Upsilon^{\dag}M, \label{5403}
%P=&Q+A'PA-M'\Upsilon^\dag M,
\end{align}
where
\begin{align}
\Upsilon=&B'RB,\label{5405}\\
M=&B'P.\label{5404}
\end{align}

Now, consider the following system without disturbances, i.e., the standard linear quadratic regulation problem:
\begin{equation}
\left\{
\begin{aligned}\label{5601}
&min  \bar{J}=\frac{1}{2}\int_{0}^{\infty}[x_{t}'Qx_{t}+u_{t}'B'RBu_{t}]\mathrm{d}t,\\
&s.t. \quad \dot{x}_{t}=Ax_{t}+Bu_{t},
\end{aligned}
\right.
\end{equation}
where $Q$ and $R$ are both semi-positive definite.

The result of the above problem can be expressed as follows.
%Recall the problem of finding $u_{t}$ to stabilize the system (\ref{5601}) and minimize the cost functional (\ref{5604}), which can be expressed as follows.
\begin{lem}\label{l56}\cite{qi2018stabilization}
Suppose that Assumption \ref{a52} holds and the system (\ref{5601}) can be stabilized if and only if the GRDE (\ref{2.19}) converges when $T\rightarrow \infty$ (i.e. $\lim_{T\rightarrow \infty}P_{t}(T)=P$). Furthermore, $P$ is a solution of the GARE (\ref{5403})–(\ref{5404}) and $P\geq 0$. In this case, the stabilizing controller is
\begin{align}
u_{t}=-\Upsilon^{\dagger}Mx_{t}, \label{5410}
\end{align}
and the optimal cost value is
\begin{align}
\bar{J}^{\ast}=\frac{1}{2}x_{0}'Px_{0}. \label{5411}
\end{align}
\end{lem}

%\begin{align}\label{2.27}
%&\dot{f}_t=M'_t\Upsilon^{-1}(B'f_t+B'REd_t)-A'f_t-P_tEd_t+Qr,\\
%&h_t=B'f_t+B'REd_t.
%\end{align}
%
%\begin{align}\label{2.28}
%&\dot{f}_t=(M'_t\Upsilon^{-1}B'f_t-A')f_t+(M'_t\Upsilon^{-1}B'RE-P_tE)d_t+Qr,\\
%&f_t=e^{(M'_t\Upsilon^{-1}B'f_t-A')}f_0+\int_{0}^{T}e^{(M'_t\Upsilon^{-1}B'f_t-A')(t-\tau)}[(M'_t\Upsilon^{-1}B'R-P_t)Ed_\tau+Qr]\mathrm{d}\tau,\\
%&h_t=B'f_t+B'REd_t.
%\end{align}

%\begin{lem}\label{l5111}
%Suppose that Assumptions \ref{a51} and \ref{a52} hold. When $T\rightarrow \infty$, the GRDE (\ref{2.19}) converges, that is, $\lim_{T\rightarrow \infty}P_{t}(T)=P$, and there is a constant $T>0$ such that $\|h_{t}(T)\|$ in (\ref{2.21}) is bounded.
%The explicit expressions for $h_{t}$ and $f_{t}$ are given as follows:
%\begin{align}
%&h_t=B'\int_{0}^{T}e^{\bar{A}_{t}(t-\tau)}(F_{\tau}d_\tau+Qr)\mathrm{d}\tau+B'REd_t, \label{2.36}\\
%&f_t=\int_{0}^{T}e^{\bar{A}_{t}(t-\tau)}(F_{\tau}d_\tau+Qr)\mathrm{d}\tau,
%\end{align}
%where
%\begin{align}
%\bar{A}_{t}=&M'_t\Upsilon^{\dag}B'-A'.\notag
%\end{align}
%\end{lem}
%

We present the main results in this section based on the above analysis.
{\theorem \label{t2} \ \
Under Assumptions \ref{a51} and \ref{a52}, if the GARE (\ref{5403})-(\ref{5404}) has a semi-positive definite solution $P$, then the system (\ref{eq.1}) is bounded.
Under such conditions, the optimal stabilizing solution can be derived as
\begin{align}
u_{t}=-\Upsilon^{\dagger}Mx_{t}-\Upsilon^{\dagger}h_{t}.
 \label{5406}
\end{align}}
\begin{proof}
The proofs of Theorem \ref{t2} can be found in Appendix B.
\end{proof}

\begin{remark}
According to the above analysis, we conclude that the system is stable under the proposed controller. And recall that solving Problem 1 can minimize the performance index \eqref{eq.2} (i.e. $Bu_t+Ed_t$ is minimized), which means that the proposed controller guarantees the disturbance rejection performance.
\end{remark}
\begin{remark}
From the above analysis process, it can be seen that the solution of the uncontrollable problem is due to the application of optimal control with a new quadratic performance index to stabilize the uncontrollable but detectable system and minimize the impact of disturbance.
\end{remark}

\begin{remark}
The key idea of this study is to transform mismatched disturbance rejection control into an LQT problem. Accordingly, in contrast to the method in \cite{gandhi2020hybrid,2012Generalized,castillo}
, where the system must be controllable, the stabilization result in Theorem \ref{t2} is obtained only under the detectable assumption.
\end{remark}

\begin{remark}\label{rem.9}
Faced with a scenario in which a disturbance is only available at the current moment, we can using the receding-horizon control method to design a controller to handle such a disturbance. The specific controller design is as follows:
\begin{align}\label{f3.46}
u_t=-\Upsilon^{-1}M_{t}x_t-\Upsilon^{-1}h_{t},
\end{align}
where $\Upsilon$, $M_{t}$, and $h_{t}$ satisfy the following equations at time $t$:
\begin{equation}\label{f3.47}
\left\{
\begin{aligned}
&\dot{P}_s=M'_s\Upsilon^{-1}M_s-P_sA-A'P_s-Q,\\
&\dot{f}_s=M'_s\Upsilon^{-1}h_s-A'f_s-P_sEd+Qr,\\
&\Upsilon=B'RB,\\
&h_s=B'f_s+B'REd, \qquad\qquad\qquad\qquad   t\leq s \leq t+\tau\\
&d=d_{t},\\
&M_s=B'P_s,\\
&f_{t+\tau}=P_{t+\tau}r.
\end{aligned}
\right.
\end{equation}
%\begin{align}
%u_{k}=-\Upsilon^{-1}_{k}M_{k}x_{k}-\Upsilon^{-1}_{k}h_{k}
%\end{align}
%for $s\in (,)$ ($T$ is a finite positive integer), where $h_{s}$, $\Upsilon_{s}$, and $M_{s}$ satisfy the following backward equations at time $k$:
%
%\begin{eqnarray}
%\left\{
%\begin{array}{lll}
%h_{s}=B'(R+P_{s+1})Ed+B'f_{s+1},\\
%f_{s}=A'P_{s+1}Ed+A'f_{s+1}-M'_{s}\Upsilon^{-1}_{s}h_{s}-Qr,\\
%\Upsilon_{s}=B'(R+P_{s+1})B,\\
%M_{s}=B'P_{s+1}A,\\
%d=d_{t},\\
%f_{T+k+1}=-P_{T+k+1}r.
%\end{array}
%\right.\label{f3.47}%\label{f2.9}
%\end{eqnarray}
\end{remark}
%\begin{remark}
%The disturbance rejection controller presented in this paper is based on a known disturbance. In the future, we will further study the design of a disturbance rejection controller with mismatched unknown disturbances (uncertainties) based on the concepts and methods presented in this paper.
%\end{remark}

\section{Numerical Examples}\label{sec5}
In this section, two examples are presented from different perspectives, illustrating the effectiveness of the proposed controller. The first example emphasizes the mismatched attenuation effect of the proposed method in an uncontrolled system compared to PID control. The second example is the application of the proposed method compared to PID control and GESOBC in an PMDC for disturbance rejection.
%In this section, four examples illustrating the effectiveness of the proposed controller are presented from different perspectives. The first example highlights the mismatched disturbance rejection effect of the proposed method in an uncontrolled system, and the second and third examples compare the disturbance rejection effect of GESOBC to that of the method proposed in this study for controlled systems with time-invariant and time-varying disturbances, respectively. The final example is the disturbance rejection application of the proposed method to an aero-engine model.

\subsection{Example A: Disturbance rejection for an uncontrollable system}
In the case where a system is detectable but uncontrollable, compare the PID control to verify the effectiveness of the control law (\ref{f3.46}) in rejecting mismatched disturbance. Consider the system (\ref{eq.1}) with the following parameters:
%In the case where a system is stable but uncontrollable, the control law in (\ref{f3.46}) is verified to reject a mismatched disturbance.
%Consider the system in (\ref{eq.1}) with the following parameters:
\begin{equation}
\begin{aligned}\label{4.3}
&{A}=\begin{bmatrix}-4&0&0\cr 0&3&1\cr0&-2&-1\end{bmatrix}, {B}=\begin{bmatrix}0\cr 1\cr 0\end{bmatrix},{E}=\begin{bmatrix}0\cr0\cr1\end{bmatrix}, {c}_o=\begin{bmatrix}0&0&1\end{bmatrix}.\notag
\end{aligned}
\end{equation}
\begin{remark}
It is trivial to determine that the state $x^1$ in the above system is stable but not controllable, and we attempt to demonstrate the superiority of our proposed controller.
\end{remark}

The initial state of the system is $x_0 =\begin{bmatrix}1&1&0\end{bmatrix}'$ and the disturbance $d=3$ acts on the system from $t = 0.5s$. The controller aims to remove the disturbance from the regulated state $x^3=c_ox$. In the proposed control law (\ref{f3.46}), the horizon $\tau$ is set to $0.05s$ and the terminal condition $P_{t+\tau}=O_{3\times3}$.
The reference $c_or$ is set to $0$ according to the goals defined above. The weight matrice are selected as $\bar Q=10^4$ and $R=I_{3\times3}$. $P_t$, $f_t$, $h_t$, $M_t$, and $\Upsilon$ can be calculated according to (\ref{f3.47}) at every time instance $t$ to derive $u_{t}$. The proposed control method was compared to PID control to track the reference and disturbance rejection effects, and the parameters in the PID method were set to $K_p=80$, $K_i=400$, and $K_d=10$ through optimal tuning. The simulation results for Example A are presented in Fig. \ref{fig_1}. %The simulation trajectories of the regulated state $x^2$ and the states $x^1$ as well as $x^3$  The trajectory of the disturbance is shown in Fig. \ref{fig_1}(d).

\begin{figure}[htbp]
    \begin{center}
%        \centering
        \includegraphics[width=9.0cm,height=8cm]{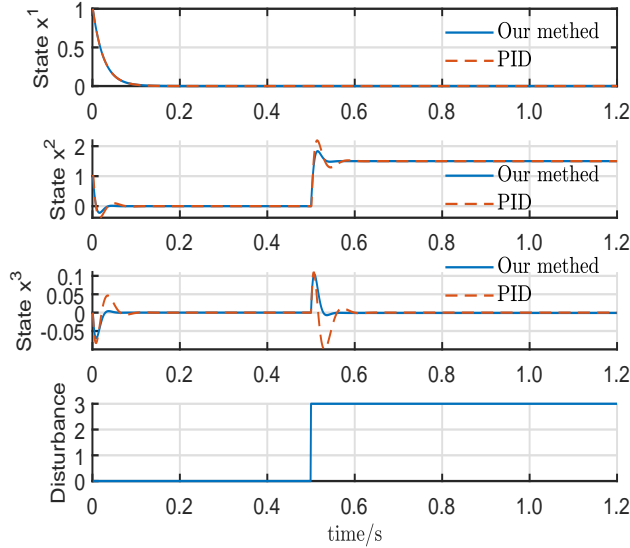}
%        \centerline{(d) Trajectory of the disturbance.}
    \end{center}
    \caption{Simulation result of Example A}\label{fig_1}
\end{figure}
In Fig. \ref{fig_1}, one can see that state $x^3$ achieves disturbance rejection quickly and stabilization is achieved for the uncontrollable state $x^1$. Therefore, the proposed method is effective for the disturbance rejection than PID control of uncontrollable systems with mismatched disturbances.

\subsection{Example B: Application to an PMDC system}
When this method is applied to the PMDC speed disturbance rejection, it is usually necessary to quickly stabilize the speed, that is, to eliminate the influence of load changes on the speed. In this case, the load variation with respect to the voltage (control input) is a mismatched disturbance. Existing PID control and GESOBC methods are mostly used for PMDC system, but both are difficult to achieve a balance between disturbance rejection and optimal performance.  Therefore, the proposed method is applied to the disturbance rejection of the speed of PMDC, and compared with the effect of PID control and GESOBC to demonstrate the effectiveness of the proposed method.

Consider PMDC system \cite{krause2013analysis} to demonstrate the effectiveness of load (disturbance) variation on speed disturbance immunity, the target system is defined as follows
%In an aero-engine nozzle performance test, it is necessary to test the effects of changes in nozzle area on the engine performance parameters. It is typically required that the engine rotor speed and other states be stabilized rapidly. In this scenario, the change in nozzle area for the mass fuel flow control loop can be considered as a mismatched disturbance. The existing PID control method is mostly used for aero-engines, and it is difficult to balance disturbance rejection with optimal performance. Therefore, the proposed method was applied to the control of an aero-engine and the results were compared to the effect of PID control to demonstrate the effectiveness of the proposed method. For considering the effects of the nozzle area change of a mixed-exhaust turbofan engine system \cite{peng2016} on engine performance to verify the effectiveness of the proposed disturbance rejection method, the target system is defined as follows:
\begin{align}{}\label{f5.1}
\dot{x}(t)=Ax(t)+{B}{u}(t)+E{d}(t)
%{y(t)=Cx(t)},\\
%{({y_o})(t)=c_ox(t)},\notag
\end{align}
with the coefficient matrix
\begin{align}\label{f5.2}
&{A}=\begin{bmatrix}-0.42&106.16\cr41.67&41.67\end{bmatrix},B=\begin{bmatrix}0\cr 83.33\end{bmatrix},E=\begin{bmatrix}-212.31\cr 0\end{bmatrix},{c}_o=\begin{bmatrix}1&0\end{bmatrix}.
\end{align}
where $x(t)=[w_m\quad i_a]'$ represents the state variables, ${y}(t)=[w_m\quad i_a]'$ is the control output, ${u}(t)= V_a(t)$ is the control input, ${d}(t)=T_L(t)$ represents the disturbance caused by load torque changes, and $w_m$ is the regulated output. $V_a(t)$ is the armature voltage $(V)$, $T_L(t)$ is load torque $(Nm)$, $w_m$  is angular speed $(rad/s)$; $i_a(t)$  is armature current $(A)$.

The disturbance in the PMDC system is the measurable load torque. The control strategy in this paper can be utilized to study the speed control on PMDC system under variable loads.  Disturbance change is $5Nm$ loaded from $0.6s$ action. From (\ref{f5.1}) and (\ref{f5.2}), one can see that the coefficient ratios of the disturbance and control inputs into different channels of the system are different with $rank(B,E)>rank(E)$. In other words, the disturbances in the system are mismatched. The control objective is to eliminate the disturbance in the output angular speed under variable loads, that is, the angular speed is maintained at 60 $rad/s$ when the load changes. In the disturbance rejection control method proposed in Theorem \ref{t1}, the terminal time $T$ is set $1.2s$ and the terminal condition $P_{T}=O$ (zero matrix). The weight matrice are selected as $\bar Q=10^4$ and $R=I_{3\times3}$. $P_t$, $f_t$, $h_t$, $M_t$, and $\Upsilon$ can be calculated using (\ref{eq.4}) to obtain $u_{t}$. The proposed control method was compared to PID control to track the reference and disturbance rejection effects. According to \cite{2012Generalized}, the state feedback matrix is chosen as $[-0.8 -0.5]$ and the disturbance compensation gain of the GESOBC method is calculated as $K_d=2$. The parameters in the PID method were set to $K_p=0.1$, $K_i=2.5$, and $K_d=0.09$ through optimal tuning.

The response curves of the PMDC system obtained using the proposed method, PID and GESOBC are presented in Fig. \ref{fig_4}. %The control input fuel flow $w_f$ variation is illustrated in Fig. \ref{fig_8}. 
Fig. \ref{fig_4} reveals that the proposed control method achieves a fast and smooth transition from the set point when the disturbance variation. Therefore, it can be concluded that the proposed method is superior to GESOBC and PID control for excellent disturbance rejection performance can reduce the impact on the PMDC system. 

These results demonstrate that the proposed method achieves satisfactory performance in terms of mismatched disturbance rejection. In this case, the proposed method is more effective than the GESOBC and PID control.
\begin{figure}[htbp]
    \begin{center}      
        \includegraphics[width=9.0cm,height=6cm]{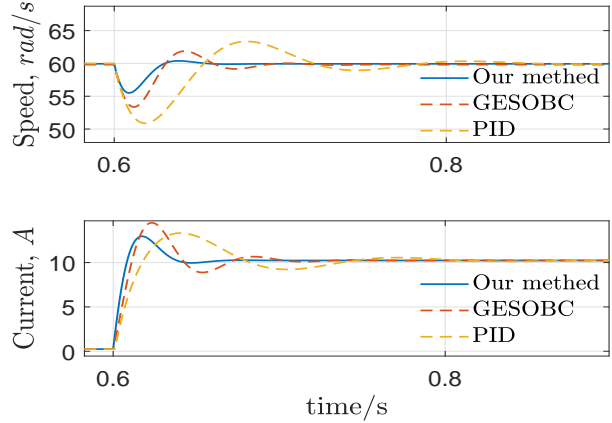}
        \caption{Simulation result of Example B}\label{fig_4}
    \end{center}
\end{figure}

\section{Conclusion}\label{sec6}
By introducing a new quadratic performance index such that the mismatched disturbance rejection problem was transformed into an LQT problem, the regulated state to track a reference and minimize the impact of disturbance.
The disturbance rejection controller and the necessary and sufficient conditions for the solvability of the problem are obtained by solving a FBDE over a finite horizon. 
In the case of an infinite horizon, a sufficient condition for system stability is obtained under detectable condition.
It is noteworthy that this approach weakens the assumption of controllability. Several examples were presented to illustrate the effectiveness of the proposed method. Although our proposed controller exhibited excellent performance, it requires known disturbances. In the future, we will further study the design of a mismatched unknown disturbance (uncertainty) rejection controller based on the concepts and methods presented in this paper.

\section{Appendix}\label{sec6}
\subsection{Appendix A: Proof of Theorem \ref{t1}}
\begin{proof}
“Necessity”: When Problem \ref{problem} admits a unique
solution, we prove that $\Upsilon=B'RB>0$.

%Using the Hamiltonian \eqref{eq.5}, the optimal co-state equation becomes %obtain the control equation from

Assume that 
\begin{align}\label{eq.3.222}
\lambda_t=P_tx_t+f_t.
\end{align}
Substituting the above formula into \eqref{f2.10}, one can yields that
\begin{align}
B'RBu_t+B'REd_t+B'P_tx_t+B'f_t=0.  \label{eq.6}
\end{align}

From the solvability of Problem 1, the regular condition $\Upsilon\Upsilon^{\dagger}V_t=V_t$ $(V_t=-B'P_tx_t-B'REd_t-B'f_t)$ is introduced. Thus, one has%According to $\Upsilon>0$, it is derived 
\begin{align}\label{eq.3.1}
u_t&=-\Upsilon^{\dagger}(B'P_tx_t+B'REd_t+B'f_t)+(I-\Upsilon\Upsilon^{\dagger})L_t,
\end{align}
where $L_t$ is an arbitrary matrix.

Let $M_t=B'P_t$ and $h_t=B'f_t+B'REd_t$, then
\begin{align}\label{eq.3.22}
u_t=-\Upsilon^{\dagger}M_tx_t-\Upsilon^{\dagger}h_t-(I-\Upsilon\Upsilon^{\dagger})L_t.
\end{align}
By differentiating both sides of equation \eqref{eq.3.222} and combining \eqref{f2.10} and \eqref{eq.3.22}, one has
\begin{align}\label{eq.8}
\dot{\lambda}_t=&\dot{P}_tx_t+P_t\dot{x}_t+\dot{f}_t\notag\\
%=&-\dot{P}_tx_t-P_t({Ax}_t+{B} {u}_t +{E}d_t)-\dot{f}_t\notag\\
=&\dot{P}_tx_t+P_t{Ax}_t+P_t{B} {u}_t+P_t{E}d_t+\dot{f}_t\notag\\
%=&-\dot{P}_tx_t-P_t{Ax}_t+M'_t\Upsilon^{-1}(B'(P_tx_t+f_t)+B'REd_t) -P_t{E}d_t-\dot{f}_t\notag\\
=&\dot{P}_tx_t+P_t{Ax}_t-M'_t\Upsilon^{\dagger}M_tx_t-M'_t\Upsilon^{\dagger}h_t-M'_t(I-\Upsilon\Upsilon^{\dagger})L_t+P_t{E}d_t+\dot{f}_t.
\end{align}
Combined with the \eqref{f2.10} and the arbitrariness of $x_{t}$, it yields
\begin{align}
\dot{P}_t&=M'_t\Upsilon^{\dagger}M_t-P_tA-A'P_t-Q,\label{eq.8.11}\\
\dot{f}_t&=M'_t\Upsilon^{\dagger}h_t+M'_t(I-\Upsilon\Upsilon^{\dagger})L_t-A'f_t-P_tEd_t+Qr.\label{eq.8.12}
\end{align}
From \eqref{f2.10}, the following two boundary value conditions can be derived
$$x_0=x,f_T=-P_Tr.$$

Combining \eqref{eq.3.222}, \eqref{eq.8.11}, and \eqref{eq.8.12} with FBDEs \eqref{f2.10}, it is not difficult to find that \eqref{eq.3.222} is the solutions of FBDEs.

Note that \eqref{eq.3.22} and  $\Upsilon=B'RB\ge0$, we know that the $u_t$ is uniquely solvable if and only if $\Upsilon$ is invertible (that is, $\Upsilon>0$). Therefore, $\Upsilon^{\dagger}$ in \eqref{eq.3.22}, \eqref{eq.8.11}, and \eqref{eq.8.12} can be replaced by $\Upsilon^{-1}$ yields that:
\begin{align}
u_t&=-\Upsilon^{-1}M_tx_t-\Upsilon^{-1}h_t,\label{eq.3.23}\\
\dot{P}_t&=M'_t\Upsilon^{-1}M_t-P_tA-A'P_t-Q,\label{eq.8.111}\\
\dot{f}_t&=M'_t\Upsilon^{-1}h_t-A'f_t-P_tEd_t+Qr.\label{eq.8.122}
\end{align}
Thus, the necessity is proved.

“Sufficiency”: 
%It can be seen that \eqref{eq.8.11} and \eqref{eq.8.12} are the ARE to be derived.
%Finally, prove the relation \eqref{eq.3} of the controller. For this, the linear relationship \eqref{eq.3.222} is brought into \eqref{eq.3.1}, and it can be proved that:
%\begin{align}\label{eq.333}
%u_t&=-\Upsilon^{-1}(B'P_tx_t+B'f_t+B'REd_t).
%\end{align}
Introduce the identity as follows
\begin{align}\label{2.13}
\frac{1}{2}&x'_TPx_T+{x}_T'f_T+H_T-\frac{1}{2}x'_0P_0x_0-{x}'_0f_0-H_0=\int_{0}^{T}\frac{\mathrm{d}}{\mathrm{d}t}[\frac{1}{2}x_t'P_tx_t+x_t'f_t+H_t]\mathrm{d}t.
\end{align}
Furthermore, using \eqref{f2.10}, \eqref{eq.8.111}, and \eqref{eq.8.122}, the above formula can be rewritten
\begin{align}\label{2.13.1}
\frac{1}{2}&x'_TPx_T+{x}_T'f_T+H_T-\frac{1}{2}x'_0P_0x_0-{x}'_0f_0-H_0\notag\\
=&\int_{0}^{T}[\frac{1}{2}\dot{x}_t'P_tx_t+\frac{1}{2}x_t'\dot{P}_tx_t+\frac{1}{2}x_t'P_t\dot{x}_t+\dot{x}_t'f_t+x_t'\dot{f}_t+\dot{H}_t]\mathrm{d}t\notag\\
%=&\int_{0}^{T}[\frac{1}{2}(Ax_t+Bu_t+Ed_t)'P_tx_t+\frac{1}{2}x_t'(M'_t\Upsilon^{-1}M_t-P_tA-A'P_t-Q)x_t\notag\\
%&+\frac{1}{2}x_t'P_t(Ax_t+Bu_t+Ed_t)+(Ax_t+Bu_t+Ed_t)'f_t\notag\\&+x_t'(M'_t\Upsilon^{-1}h_t-A'f_t-P_tEd_t+Qr)+\frac{1}{2}h'_t\Upsilon^{-1}h_t-d'_tE'f_t-\frac{1}{2}r'Qr-\frac{1}{2}d'_tE'REd_t]\mathrm{d}t\notag\\
%=&\int_{0}^{T}[\frac{1}{2}u'_tM'_tx_t+\frac{1}{2}x_t'(M'_t\Upsilon^{-1}M_t-Q)x_t+\frac{1}{2}x_t'M'_tu\notag\\
%&+(Bu_t+Ed_t)'f_t+x_t'(M'_t\Upsilon^{-1}h_t+Qr)+\frac{1}{2}h'_t\Upsilon^{-1}h_t-d'_tE'f_t-\frac{1}{2}r'Qr-\frac{1}{2}d'_tE'REd_t]\mathrm{d}t\notag\\
=&\int_{0}^{T}[\frac{1}{2}u'_t\Upsilon u_t+u'_tM_tx_t+\frac{1}{2}x_t'M'_t\Upsilon^{-1}M_tx_t+u'_th_t-u'_tB'REd_t+d'_tE'f_t\notag\\
&+x_t'M'_t\Upsilon^{-1}h_t+x_t'Qr+\dot{H}_t-\frac{1}{2}x_t'Qx_t-\frac{1}{2}u'_t\Upsilon u_t]\mathrm{d}t.%%
\end{align}
Putting \eqref{eq.3.11} into the above equation and noting the terminal condition $f_T=-P_Tr$, one has
\begin{align}\label{2.15}
J=&\frac{1}{2}x'_0P_0x_0+{x}'_0f_0+H_0\notag\\
&+\int_{0}^{T}\frac{1}{2}[u_t+\Upsilon^{-1}M_tx_t+\Upsilon^{-1}h_t]'\Upsilon[u_t+\Upsilon^{-1}M_tx_t+\Upsilon^{-1}h_t]\mathrm{d}t.%%
\end{align}
The unique solvability of Problem \ref{problem} is easily derived based on the positive definiteness of $\Upsilon$.
%where 
%\begin{align}\label{2.17}
%\dot{H}_t=\frac{1}{2}h'_t\Upsilon^{-1}h_t-d'_tE'f_t-\frac{1}{2}r'Qr-\frac{1}{2}d'_tE'REd_t
%\end{align}

Sufficiency is demonstrated. The proof is complete.

\end{proof}
\subsection{Appendix B: Proof of Theorem \ref{t2}}
\begin{proof} \ Suppose the GARE (\ref{5403})-(\ref{5405}) has a solution $P\geq 0$. We will prove the bounded stabilization of the system (\ref{eq.1}). The following are explicit expressions for $h_{t}$ and $f_{t}$:
\begin{align}
&h_t(T)=B'\int_{0}^{T}e^{\bar{A}_{t}(T)(t-\tau)}F_{\tau}(T)d_\tau\mathrm{d}\tau+B'\mathcal{R}_{t}(T)r+B'REd_t, \label{51102}\\
&f_t(T)=\int_{0}^{T}e^{\bar{A}_{t}(t-\tau)}F_{\tau}(T)d_\tau\mathrm{d}\tau+\mathcal{R}_{t}(T)r,\label{51103}
\end{align}
where
\begin{align}
\bar{A}_{t}(T)=&A'-M'_t(T)\Upsilon^{-1}B',\notag\\
F_{t}(T)=&(M'_t(T)\Upsilon^{-1}B'R-P_t)E,\label{51104}\\
\mathcal{R}_{t}(T)=&\int_{0}^{T}e^{\bar{A}_{t}(T)(t-\tau)}Q\mathrm{d}\tau, \label{51105}\\  
{f}_{T}(T)=&-P_{T}(T)r.\notag
\end{align}

%\end{remark}

First, we demonstrate the boundedness of $h_{t}$. From Lemma \ref{l56}, $\rho(A-B\Upsilon^{\dagger}M)<0$.
%Hence, the invertibility of $I-(A-B\Upsilon^{\dagger}M)$ can be guaranteed.
Considering the convergence of $P_{t}(T)$, we can find that $F_{t}(T)$ in \eqref{51104} and $\mathcal{R}_{t}(T)$ in \eqref{51105} are convergent. Therefore, there exist constants $\mathcal{C}_{F},  \mathcal{C}_{\mathcal{R}}$ satisfying $F_{t}\leq \mathcal{C}_{F}, \mathcal{R}_{t}\leq \mathcal{C}_{\mathcal{R}}$.

From \eqref{51102} and Assumption \ref{a51}, we have
\begin{align}
\lVert\int_{0}^{\infty}e^{\bar{A}_{t}(t-\tau)}&F_{\tau}d_\tau\mathrm{d}\tau\lVert\leq\int_{0}^{\infty}\lVert e^{\bar{A}_{t}(t-\tau)}	\lVert 	\lVert F_{\tau}	\lVert	\lVert d_\tau \lVert	\mathrm{d}\tau\leq \mathcal{C}_{F} \bar{d}\int_{0}^{\infty}\lVert e^{\bar{A}_{t}(t-\tau)}	\lVert	\mathrm{d}\tau\leq \mathcal{C},
\end{align}
where $\rho(\bar{A}_{t})<0$ guarantees the boundedness of  $\int_{0}^{\infty}\lVert e^{\bar{A}_{t}(t-\tau)}	\lVert	\mathrm{d}\tau$. Therefore, $h_{t}$ is bounded. Similarly, $f_{t}$ is bounded.

Note that when $u_{t}=-\Upsilon^{\dagger}Mx_{t}-\Upsilon^{\dagger}h_{t}$,  we obtain
\begin{align}
\dot{x}_{t}=(A-B\Upsilon^{\dagger}M)x_{t}-B\Upsilon^{\dagger}h_{t}.\label{5408}
\end{align}
From the boundedness of $h_{t}$ and $\rho(A-B\Upsilon^{\dagger}M)<0$, we find that the system (\ref{eq.1}) has bounded stability with $u_{t}=-\Upsilon^{\dagger}Mx_{t}-\Upsilon^{\dagger}h_{t}$.
\end{proof}

\bibliographystyle{ieeetr}
\bibliography{reference}

\begin{thebibliography}{10}

\bibitem{chen2015disturbance}
W.-H. Chen, J.~Yang, L.~Guo, and S.~Li, ``Disturbance-observer-based control
  and related methods—{A}n overview,'' {\em IEEE Transactions on Industrial
  Electronics}, vol.~63, no.~2, pp.~1083--1095, 2015.

\bibitem{li2014disturbance}
S.~Li, J.~Yang, W.-H. Chen, and X.~Chen, {\em Disturbance observer-based
  control: methods and applications}.
\newblock CRC Press, 2014.

\bibitem{GUO20132911}
B.-Z. Guo and F.-F. Jin, ``The active disturbance rejection and sliding mode
  control approach to the stabilization of the euler--bernoulli beam equation
  with boundary input disturbance,'' {\em Automatica}, vol.~49, no.~9,
  pp.~2911--2918, 2013.

\bibitem{4796887}
J.~Han, ``From {PID} to active disturbance rejection control,'' {\em IEEE
  transactions on Industrial Electronics}, vol.~56, no.~3, pp.~900--906, 2009.

\bibitem{5572931}
Y.~Huang, W.~Xue, and X.~Yang, ``Active disturbance rejection control:
  Methodology, theoretical analysis and applications,'' in {\em Proceedings of
  the 29th {C}hinese Control Conference}, pp.~6083--6090, 2010.

\bibitem{ZHAO2014882}
S.~Zhao and Z.~Gao, ``Modified active disturbance rejection control for
  time-delay systems,'' {\em ISA transactions}, vol.~53, no.~4, pp.~882--888,
  2014.

\bibitem{shtessel2014sliding}
Y.~Shtessel, C.~Edwards, L.~Fridman, A.~Levant, {\em et~al.}, {\em Sliding mode
  control and observation}, vol.~10.
\newblock Springer, 2014.

\bibitem{young1999control}
K.~D. Young, V.~I. Utkin, and U.~Ozguner, ``A control engineer's guide to
  sliding mode control,'' {\em IEEE Transactions on Control Systems
  Technology}, vol.~7, no.~3, pp.~328--342, 1999.

\bibitem{chen2003nonlinear}
W.-H. Chen, ``Nonlinear disturbance observer-enhanced dynamic inversion control
  of missiles,'' {\em Journal of Guidance, Control, and Dynamics}, vol.~26,
  no.~1, pp.~161--166, 2003.

\bibitem{chwa2004compensation}
D.~Chwa, J.~Y. Choi, and J.~H. Seo, ``Compensation of actuator dynamics in
  nonlinear missile control,'' {\em IEEE Transactions on Control Systems
  Technology}, vol.~12, no.~4, pp.~620--626, 2004.

\bibitem{mohamed2007design}
Y.~A.~I. Mohamed, ``Design and implementation of a robust current-control
  scheme for a {PMSM} vector drive with a simple adaptive disturbance
  observer,'' {\em IEEE Transactions on Industrial Electronics}, vol.~54,
  no.~4, pp.~1981--1988, 2007.

\bibitem{isidori1985nonlinear}
A.~Isidori, {\em Nonlinear control systems: an introduction}.
\newblock Springer, 1985.

\bibitem{castillo}
A.~Castillo, P.~Garc{\'\i}a, R.~Sanz, and P.~Albertos, ``Enhanced extended
  state observer-based control for systems with mismatched uncertainties and
  disturbances,'' {\em ISA Transactions}, vol.~73, pp.~1--10, 2018.

\bibitem{chen2016adrc}
S.~Chen, W.~Bai, and Y.~Huang, ``Adrc for systems with unobservable and
  unmatched uncertainty,'' in {\em Proceedings of the 35th {C}hinese control
  conference}, pp.~337--342, IEEE, 2016.

\bibitem{ginoya2013sliding}
D.~Ginoya, P.~Shendge, and S.~Phadke, ``Sliding mode control for mismatched
  uncertain systems using an extended disturbance observer,'' {\em IEEE
  Transactions on Industrial Electronics}, vol.~61, no.~4, pp.~1983--1992,
  2013.

\bibitem{2012Generalized}
S.~Li, J.~Yang, W.-H. Chen, and X.~Chen, ``Generalized extended state observer
  based control for systems with mismatched uncertainties,'' {\em IEEE
  Transactions on Industrial Electronics}, vol.~59, no.~12, pp.~4792--4802,
  2011.

\bibitem{9339876}
Z.-H. Wu, F.~Deng, B.-Z. Guo, C.~Wu, and Q.~Xiang, ``Backstepping active
  disturbance rejection control for lower triangular nonlinear systems with
  mismatched stochastic disturbances,'' {\em IEEE Transactions on Systems, Man,
  and Cybernetics: Systems}, vol.~52, no.~4, pp.~2688--2702, 2021.

\bibitem{yang2012nonlinear}
J.~Yang, S.~Li, and W.-H. Chen, ``Nonlinear disturbance observer-based control
  for multi-input multi-output nonlinear systems subject to mismatching
  condition,'' {\em International Journal of Control}, vol.~85, no.~8,
  pp.~1071--1082, 2012.

\bibitem{6129407}
J.~Yang, S.~Li, and X.~Yu, ``Sliding-mode control for systems with mismatched
  uncertainties via a disturbance observer,'' {\em IEEE Transactions on
  Industrial Electronics}, vol.~60, no.~1, pp.~160--169, 2012.

\bibitem{gandhi2020hybrid}
R.~V. Gandhi and D.~M. Adhyaru, ``Hybrid extended state observer based control
  for systems with matched and mismatched disturbances,'' {\em ISA
  Transactions}, vol.~106, pp.~61--73, 2020.

\bibitem{yang1994disturbance}
W.-C. Yang and M.~Tomizuka, ``{Disturbance rejection through an external model
  for nonminimum phase systems},'' {\em Journal of Dynamic Systems,
  Measurement, and Control}, vol.~116, pp.~39--44, 03 1994.

\bibitem{10.1115/1.2896180}
M.~Tomizuka, K.~Chew, and W.~Yang, ``Disturbance rejection through an external
  model,'' {\em Journal of Dynamic Systems, Measurement, and Control},
  vol.~112, pp.~559--564, 12 1990.

\bibitem{9050185}
T.~Pengliang, ``Feedback linearization of mimo nonlinear system with measurable
  disturbance,'' in {\em Proceedings of the 12th International conference on
  measuring technology and mechatronics automation}, pp.~744--749, 2020.

\bibitem{GUO201941}
H.~Guo, D.~Cao, H.~Chen, Z.~Sun, and Y.~Hu, ``Model predictive path following
  control for autonomous cars considering a measurable disturbance:
  Implementation, testing, and verification,'' {\em Mechanical Systems and
  Signal Processing}, vol.~118, pp.~41--60, 2019.

\bibitem{qi2018stabilization}
Q.~Qi, H.~Zhang, and Z.~Wu, ``Stabilization control for linear continuous-time
  mean-field systems,'' {\em IEEE Transactions on Automatic Control}, vol.~64,
  no.~8, pp.~3461--3468, 2018.

\bibitem{krause2013analysis}
P.~C. Krause, O.~Wasynczuk, S.~D. Sudhoff, and S.~D. Pekarek, {\em Analysis of
  electric machinery and drive systems}, vol.~75.
\newblock John Wiley \& Sons, 2013.

\end{thebibliography}
\end{document}